\documentclass[10pt, a4paper]{article}

\usepackage[utf8]{inputenc}
\usepackage{amsmath}
\usepackage{amssymb}
\usepackage{enumerate}
\usepackage{amsthm}
\usepackage{dsfont}
\usepackage{setspace}
\usepackage[unicode]{hyperref}

\newtheorem{theorem}{Theorem}

\newtheorem{proposition}[theorem]{Proposition}

\newtheorem{assumption}[theorem]{Assumption}
\newtheorem*{assumption*}{Assumption}
\theoremstyle{definition}\newtheorem{remark}[theorem]{Remark}

\newcommand{\R}{\mathbb{R}}
\newcommand{\N}{\mathbb{N}}

\newcommand{\twp}{\hat{u}}
\newcommand{\tw}{u^{TW}}

\title{Finite-Size Effects on Traveling Wave Solutions to Neural Field Equations}

\author{Eva Lang \footnotemark[1]\ \footnotemark[2] \and Wilhelm Stannat \footnotemark[1]\ \footnotemark[2]}

\begin{document}

\maketitle
\renewcommand{\thefootnote}{\fnsymbol{footnote}}
\footnotetext[1]{Institut f\"ur Mathematik, Technische Universit\"at Berlin, D-10623 Berlin, Germany (lang@math.tu-berlin.de, stannat@math.tu-berlin.de)}
\footnotetext[2]{Bernstein Center for Computational Neuroscience, D-10115 Berlin, Germany. }
\renewcommand{\thefootnote}{\arabic{footnote}}

\begin{abstract}
Neural field equations are used to describe the spatiotemporal evolution of the activity in a 
network of synaptically coupled populations of neurons in the continuum limit. Their heuristic 
derivation involves two approximation steps. Under the assumption that each population in the 
network is large, the activity is described in terms of a population average. The discrete 
network is then approximated by a continuum. In this article we make the two approximation steps 
explicit. Extending a model by Bressloff and Newby, we describe the evolution of the activity in 
a discrete network of finite populations by a Markov chain. In order to determine finite-size 
effects - deviations from the mean field limit due to the finite size of the populations in the 
network - we analyze the fluctuations of this Markov chain and set up an approximating system of 
diffusion processes. We show that a well-posed stochastic neural field equation with a noise term 
accounting for finite-size effects on traveling wave solutions is obtained as the strong 
continuum limit.
\end{abstract}

\section{Introduction}

The analysis of networks of neurons of growing size quickly becomes involved from a computational as well as from an analytic perspective when one tracks the spiking activity of every neuron in the network. It can therefore be useful to zoom out from the microscopic perspective and identify a population activity as an average over a certain group of neurons. In the heuristic derivation of such population models it is usually assumed that each of the populations in the network is infinite, such that, in the spirit of the law of large numbers, the description of the activity in each population reduces to a description of the mean.
By considering a spatially extended network and letting the density of populations go to infinity, neural field equations are obtained as the continuum limit of these models. Here we consider the voltage-based neural field equation, which is a nonlocal evolution equation of the form
\begin{equation}
\label{eq:nfe}
\frac{\partial}{\partial t} u(x,t) = -u(x,t) + w\ast F(u(\cdot, t))(x), \qquad x\in \R, t\geq 0
\end{equation}
where $u(x,t)$ describes the average membrane potential in the population at $x$ at time $t$, $w:\R \rightarrow [0,\infty)$ is a kernel describing the strengths of the synaptic connections between the populations, and the gain function $F:\R \rightarrow [0, 1]$ relates the potential to the activity in the population.

Neural field equations were first introduced by Amari \cite{amari} and Wilson and Cowan \cite{wilsoncowan72, wilsoncowan73} and have since been used extensively to study the spatio-temporal dynamics of the activity in coupled populations of neurons. While they are of a relatively simple form, they exhibit a variety of interesting spatio-temporal patterns. For an overview see for example \cite{ermentroutbook, coombesbook, ermentroutreview, bressloffreview, bressloffbook}. In this article we will concentrate on traveling wave solutions, modeling the propagation of activity, that were proven to exist in \cite{ermentroutmcleod}. 

The communication of neurons is subject to noise. It is therefore crucial to study stochastic versions of (\ref{eq:nfe}). While several sources of noise have been identified on the single neuron level, it is not clear how noise translates to the level of populations. Since neural field equations are derived as mean field limits, the usual effects of noise should have averaged out on this level. However, the actual finite size of the populations leads to deviations from the mean field behavior, suggesting finite size effects as an intrinsic source of noise.
 
The (heuristic) derivation of neural field equations involves two approximation steps. First, the local dynamics in each population is reduced to a description of the mean activity. Second, the discrete network is approximated by a continuum. 
In this article we make these two approximation steps explicit.
In order to describe deviations from the mean field behavior for finite population sizes, we set up a Markov chain to describe the evolution of the activity in the finite network, extending a model by Bressloff and Newby \cite{bressloffnewby}. The transition rates are chosen in such a way that we obtain the voltage-based neural network equation in the infinite population limit. We analyze the fluctuations of the Markov chain in order to determine a stochastic correction term describing finite-size effects. In the case of fluctuations around traveling wave solutions, we set up an approximating system of diffusion processes and prove that a well-posed stochastic neural field equation is obtained in the continuum limit.

In order to derive corrections to the neural field equation accounting for finite-size effects, in \cite{bressloffsystemsize}, Bressloff (following Buice and Cowan \cite{buicecowan}) sets up a continuous time Markov chain describing the evolution of the activity in a finite network of populations of finite size $N$. The rates are chosen such that in the limit as $N \rightarrow \infty$ one obtains the usual activity-based network equation. He then carries out a van Kampen system size expansion of the associated master equation in the small parameter $1/N$ to derive deterministic corrections of the neural field equation
in the form of coupled differential equations for the moments.
To first order, the finite-size effects can be characterized as Gaussian fluctuations around the mean field limit.

The model is considered from a mathematically rigorous perspective by Riedler and Buckwar in \cite{riedler}. They make use of limit theorems for Hilbert-space valued piecewise deterministic Markov processes recently obtained in \cite{riedlerthieullenwainrib} as an extension of Kurtz's convergence theorems for jump Markov processes to the infinite-dimensional setting. They derive a law of large numbers and a central limit theorem for the Markov chain, realizing the double limit (number of neurons per population to infinity and continuum limit)
at the same time. 
They formally set up a stochastic neural field equation, but the question of well-posedness is left open.

In \cite{bressloffnewby}, Bressloff and Newby extend the original approach of \cite{bressloffsystemsize} by including synaptic dynamics and consider a Markov chain modeling the activity coupled to a piecewise deterministic process describing the synaptic current (see also section 6.4 in \cite{bressloffbook} for a summary).
In two different regimes, the model covers the case of Gaussian-like fluctuations around the mean-field limit as derived in \cite{bressloffsystemsize}, as well as a situation in which the activity has Poisson statistics as considered in \cite{buicecowanchow}.

Here we consider the question how finite-size effects can be included in the voltage-based neural field equation. 
We take up the approach of describing the dynamics of the activity in a finite-size network by a continuous-time Markov chain and motivate a choice of jump rates that will lead to the voltage-based network equation in the infinite-population limit. 
We derive a law of large numbers and a central limit theorem for the Markov chain.
Instead of realizing the double limit as in \cite{riedler}, we split up the limiting procedure, which in particular allows us to insert further approximation steps. We follow the original approach by Kurtz to determine the limit of the fluctuations of the Markov chain. By linearizing the noise term around the traveling wave solution, we obtain an approximating system of diffusion processes.
After introducing correlations between populations lying close together (cf. section \ref{subsection:correlations}) we obtain a well-posed $L^2(\R)$-valued stochastic evolution equation, with a noise term approximating finite-size effects on traveling waves, which we prove to be the strong continuum limit of the associated network.

The article is structured as follows. We recall how population models can be derived heuristically in section \ref{section:fseinpopmodels} and summarize the work on the description of finite-size effects that can be found in the literature so far. In section \ref{section:MC} we introduce our Markov chain model for determining finite-size effects in the voltage-based neural field equation and prove a law of large numbers and a central limit theorem for our choice of jump rates. We use it to set up a diffusion approximation with a noise term accounting for finite-size effects on traveling wave solutions in section \ref{section:diffusionapproximation}. Finally, in section \ref{section:continuumlimit}, we prove that a well-posed stochastic neural field equation is obtained in the continuum limit.

\subsection*{Assumptions on the Parameters}

As usual, we take the gain function $F:\R \rightarrow [0,1]$ to be a sigmoid function, for 
example $F(x) = \frac{1}{1+ e^{-\gamma(x-\kappa)}}$ for some $\gamma > 0$, $0<\kappa<1$. In 
particular we assume that
\begin{enumerate}[(i)]
	\item $F \geq 0, \lim_{x \downarrow -\infty}F(x)=0, \lim_{x\uparrow\infty}F(x)=1$
	\item $ F(x)-x$ has exactly three zeros $0 < a_1 < a < a_2<1$
	\item $F \in \mathcal{C}^3$ and $F', F''$ and $F'''$ are bounded
	\item  $F' > 0, F'(a_1) < 1, F'(a_2) < 1, F'(a)>1$
\end{enumerate}

Our assumptions on the synaptic kernel $w$ are the following
\begin{enumerate}[(i)]
	\item $w \in \mathcal{C}^1$
	\item $w(x,y)=w(|x-y|)\geq 0$ is nonnegative and homogeneous
	\item $\int_{-\infty}^{\infty} w(x) dx = 1, w_x \in L^1$
\end{enumerate} 

Assumption (iv) on $F$ implies that $a_1$ and $a_2$ are stable fixed points of (\ref{eq:nfe}), while $a$ is an unstable fixed point.
It has been shown in \cite{ermentroutmcleod} that under these assumptions there exists a unique monotone traveling wave solution to (\ref{eq:nfe}) connecting the stable fixed points (and in \cite{chenfengxin}, that traveling wave solutions are necessarily monotone). That is, there exists a unique wave profile $\hat{u}:\R \rightarrow [0,1]$ and a unique wave speed $c \in \R$ such that $u^{TW}(t,x) = \tw_t(x) := \hat{u}(x-ct)$ is a solution to (\ref{eq:nfe}), i.e.
\begin{align*}
-c\partial_x \tw_t (x) = \partial_t \tw_t(x) & = -\tw_t(x) + \int_{-\infty}^{\infty} w(x-y) F(\tw_t(y)) dy \\
& = -\tw_t(x) + w \ast F(\tw_t)(x),
\end{align*}
and 
\[\lim_{x \rightarrow -\infty} \hat{u}(x) = a_1, \qquad \lim_{x\rightarrow\infty} \hat{u}(x) = a_2.\]  
As also pointed out in \cite{ermentroutmcleod}, we can without loss of generality assume that $c\geq 0$.
Note that $\hat{u}_x \in L^2(\R)$ since in the case $c>0$
\begin{align*}
\int \hat{u}_x^2(x) dx & = \int \hat{u}_x(x) \frac{1}{c}\left( \hat{u}(x) - w\ast F(\hat{u})(x)\right) dx \\
& \leq \frac{1}{c} \left( \|\hat{u}\|_{\infty} + \|F\|_{\infty} \int w(x) dx\right) \int \hat{u}_x(x) dx \\
& = (\frac{1}{c}a_2+1) (a_2-a_1),
\end{align*}
and in the case $c=0$,
\[\int \hat{u}_x^2(x) dx  = \int \hat{u}_x(x) \ w_x \ast F(\hat{u})(x) dx \leq \|w_x\|_1 (a_2-a_1).\]


\section{Finite-Size Effects in Population Models}
\label{section:fseinpopmodels}

\subsection{Population Models}
\label{subsection:populationmodels}

In population models, or firing rate models, instead of tracking the spiking activity of every 
neuron in the network, neurons are grouped together and the activity is identified as a 
population average. We start by giving a heuristic derivation of population models, 
distinguishing as usual between an activity-based and a voltage-based regime.

We consider a population of $N$ neurons. We say that a neuron is `active' if it is in the process of firing an action potential such that its membrane potential is larger than some threshold value $\kappa$. If $\Delta$ is the width of an action potential, then a neuron is active at time $t$ if it fired a spike in the time interval $(t-\Delta,t]$. 
We define the population activity at a given time $t$ as the proportion of active neurons,
\[a^N(t) = \frac{\mbox{\# neurons that are active at time $t$}}{N} \in \{ 0, \frac{1}{N}, \frac{2}{N}, \ldots, 1\}.\]
We assume that all neurons in the population are identical and receive the same input. 
If the neurons fire independently from each other, then for a constant input current $I$,
\[a^N(t) \xrightarrow{N\rightarrow \infty} F(I),\]
where $F(I)$ is the probability that a neuron receiving constant stimulation $I$ is active.
In the infinite population limit, the population activity is thus related to the input current via the function $F$, called the \emph{gain function}.
Sometimes one also defines $F$ as a function of the potential $u$, assuming that the potential is proportional to the current as in Ohm's law.
$F$ is typically a nonlinear function. It is usually modeled as a sigmoid, for example
\[F(x) = \frac{1}{1+e^{-\gamma (x-\kappa)}}\]
for some $\gamma>0$ and some threshold $\kappa>0$, imitating the threshold-like nature of spiking activity.

Sometimes a \emph{firing rate} is considered instead of a probability. We define the population firing rate $\lambda^N$ as
\[\lambda^{\delta,N}(t) = \frac{\mbox{\# spikes in the time interval $(t-\delta,t]$}}{\delta N}.\]
If $\delta=\Delta$, then $\Delta \lambda^{\delta,N}(t)=a^N(t)$.
At constant potential $u$, 
$\lim_{\delta\rightarrow 0} \lim_{N\rightarrow\infty} \lambda^{\delta,N}(t) = \lambda(u)$, where $\lambda(u)$ is the single neuron firing rate.
Note that $\lambda \leq \frac{1}{\Delta}$. The firing rate is related to the probability $F(u)$ via
\[F(u) \approx \lambda(u) \Delta. \]

If the stimulus varies in time, then the activity may track this stimulus with some delay such that
\[a(t+\tau_a) = F(u(t))\]
for some time constant $\tau_a$.
Taylor expansion of the left-hand side gives an approximate description of the (infinite population) activity in terms of the differential equation 
\begin{equation}
\tau_a \dot{a}(t) = -a(t)+F(u(t)),
\end{equation}
to which we refer as the \emph{rate equation}.

We now consider a network of $P$ populations, each consisting of $N$ neurons. 
We assume that each presynaptic spike in population $j$ at time $s$ causes a postsynaptic potential 
\[h(t-s)= \frac{1}{N} w_{ij} \frac{1}{\tau_m} e^{-\frac{1}{\tau_m}(t-s)}\]
in population $i$ at time $t$. 
Here the $(w_{ij})$ are weights characterizing the strength of the synaptic connections between populations $i$ and $j$, and $\tau_m$ is the membrane time constant, describing how fast the membrane potential relaxes back to its resting value.

Under the assumption that all inputs add up linearly, the potential in population $i$ at time $t$ is given as
\begin{equation*}
\label{eq:voltage1}
u_i^N(t) = \sum_{j=1}^P w_{ij} \int_{-\infty}^t \frac{1}{\tau_m} e^{-\frac{1}{\tau_m}(t-s)} a_j^N(s) ds.
\end{equation*}
In the infinite population limit we obtain
\begin{equation}
\label{eq:voltage2}
u_i(t) = \sum_{j=1}^P w_{ij} \int_{-\infty}^t \frac{1}{\tau_m} e^{-\frac{1}{\tau_m}(t-s)} a_j(s) ds,
\end{equation}
where
\begin{equation}
\label{eq:rateeq}
\tau_a \dot{a}_j(t) = -a_j(t) + F(u_j(t)).
\end{equation}
The behavior of the coupled system $(u_i, a_i)$ depends on the two time constants, $\tau_m$ and $\tau_a$. We consider two different regimes in which the model can be reduced to just one of the two variables, $u$ or $a$.

\subsubsection*{Case 1: $\tau_m \gg \tau_a \rightarrow 0$}

In this regime we can assume that the activity reacts to changes in input immediately such that $a_j(t) = F(u_j(t))$.
Then (\ref{eq:voltage2}) can be closed in the variables $u_i$ and we obtain
\begin{equation}
\label{eq:voltage3}
u_i(t) = \sum_{j=1}^P w_{ij} \int_{-\infty}^t \frac{1}{\tau_m} e^{-\frac{1}{\tau_m}(t-s)} F(u_j(s)) ds.
\end{equation}
Differentiation yields the system of ordinary differential equations
\begin{equation}
\label{eq:networkvoltage}
\frac{d}{dt} u_i(t) = \frac{1}{\tau_m} \bigg( - u_i(t) + \sum_{j=1}^P w_{ij} F(u_j(t)) \bigg),
\end{equation}
which we will call the \emph{voltage-based neural network equation}.

\subsubsection*{Case 2: $\tau_a \gg \tau_m \rightarrow 0$}

By (\ref{eq:voltage2}), 
\begin{equation*}
\label{eq:activity1}
a_i(t+\tau_a) = F\bigg( \sum_{j=1}^P w_{ij} \int_{-\infty}^t \frac{1}{\tau_m} e^{-\frac{1}{\tau_m}(t-s)} a_j(s) ds \bigg).
\end{equation*}
Letting $\tau_m\rightarrow 0$ we obtain
\begin{equation*}
\label{eq:activity2}
a_i(t+\tau_a) = F\bigg( \sum_{j=1}^P w_{ij} a_j(t) \bigg).
\end{equation*}
Using again that $a_i(t+\tau_a) \approx a_i(t) + \tau_a a_i'(t)$, 
we end up with the system of ordinary differential equations
\begin{equation}
\label{eq:networkactivity}
\tau_a \frac{d}{dt} a_i(t) = -a_i(t) + F\bigg( \sum_{j=1}^P w_{ij} a_j(t) \bigg),
\end{equation}
which we will call the \emph{activity-based neural network equation}.

\subsection{Finite-Size Effects in the Literature}
In \cite{bressloffnewby}, Bressloff and Newby set up a model for the evolution of the activity in a network of finite populations. They define the activity in population $j$ as
\[a^{\delta, N}_j(t) = \frac{\mbox{\# spikes in $(t-\delta, t]$ in population $j$ }}{\delta N },\]
where $\delta$ is a time window of variable size. If $\delta$ is chosen as the width of an action potential $\Delta$, then we obtain our original notion of the activity, $\Delta a^{\Delta, N} = a^N$. Here the activity is modeled as a rate rather than a probability. Note that the number of spikes in the time interval $(t-\delta,t]$ is limited by $n_{max}:= N  \vee \big[\frac{\delta N}{\Delta} \big] $.

They describe the dynamics of $a^{\delta, N}$ by a Markov chain with state space $\{ 0, \frac{1}{\delta N}, ... ,\frac{n_{max}}{\delta N}  \}$ and jump rates
\begin{equation}
\label{eq:jumpratesact}
\begin{split}
q_a^N(x, x+ \frac{1}{\delta N} e_i) & = \frac{1}{\tau_a} \delta N  \lambda (u^N_i(t))  \ \mbox{if $x(i) < \frac{n_{max}}{\delta N}$ }\\
q_a^N(x, x- \frac{1}{\delta N} e_i) & = \frac{1}{\tau_a} \delta N x_i(t),
\end{split}
\end{equation}
where $e_i$ denotes the $i$-th unit vector, where $\lambda(u)$ is the firing rate at potential $u$, related to the probability $F(u)$ via $\Delta \lambda(u) = F(u)$, and where $u^N$ evolves according to (\ref{eq:voltage1}),
\[\dot{u}_i^N(t) = \frac{1}{\tau_m} \Big( -u_i(t) + \sum_{j=1}^P w_{ij} a_j^N(t) \Big)\]
The idea is that the activation rate should be proportional to $\lambda(u)$, while the inactivation rate should be proportional to the activity itself. 
The rates are chosen such that in the limit as $N$ goes to infinity, we obtain the neural rate equation
\[\tau_a \dot{a}_i(t) = - a_i(t) + \lambda(u_i(t)).\]

They consider two regimes. 

\subsubsection*{Case 1: $\delta=1$,  $\tau_a \gg \tau_m \rightarrow 0$}

In the first regime, the size of the time window $\delta$ is fixed, say $\delta=1$. 
If $\tau_a \gg \tau_m \rightarrow 0$, then as in section \ref{subsection:populationmodels}, $u^N_i(t) = \sum_{j=1}^P w_{ij} a_j^{\delta, N}(t)$. The description of the Markov chain can thus be closed in the variables $a^{\delta, N}_i$, leading to the model already considered in \cite{bressloffsystemsize}. In the limit $N\rightarrow\infty$ one obtains the activity-based network equation
\begin{equation}
\label{eq:fseactivitynetwork}\tau_a \dot{a}_i^N(t) = -a_i^N(t) + \lambda\Big( \sum_{j=1}^P w_{ij} a^N_j(t)\Big).
\end{equation}
By formally approximating to order $\frac{1}{N}$ in the associated master equation, they derive a stochastic correction to (\ref{eq:fseactivitynetwork}), leading to the diffusion approximation
\begin{align*}
da^{\delta, N}_i(t) 
& \approx \frac{1}{\tau_a} \bigg( - a^{\delta, N}_i(t) +  \lambda\bigg( w_{ij} \sum_{j=1}^P a^{\delta, N}_j(t)\bigg) \bigg) dt \\
& \qquad +  \frac{1}{\sqrt{\tau_a N}} \bigg( a^{\delta, N}_i(t) +   \lambda \bigg( \sum_{j=1}^P w_{ij} a^{\delta, N}_j(t)\bigg) \bigg)^{\frac{1}{2}} dB_j(t)
\end{align*}
for independent Brownian motions $B_j$.

In \cite{riedler}, Riedler and Buckwar rigorously derive a law of large numbers and a central limit theorem for the sequence of Markov chains as $N$ tends to infinity. 
Note that the nature of the jump rates is such that the process has to be `forced' to stay in its natural domain $[0,\frac{n_{max}}{N}]$ by setting the jump rate to $0$ at the boundary. As they point out, this discontinuous behavior is difficult to deal with mathematically.
They therefore have to slightly modify the model and allow the activity to be larger than $\frac{n_{max}}{N}$.
They embed the Markov chain into $L^2(D)$ for a bounded domain $D\subset \R^d$ and derive the LLN in $L^2(D)$ and the CLT in the Sobolev space $H^{-\alpha}(D)$ for some $\alpha>d$.

\subsubsection*{Case 2: $\delta =\frac{1}{N}, \tau_m \gg \tau_a$}

In the second regime, the size of the time window $\delta$ goes to $0$ as $N$ goes to infinity such that $\delta N=1$.
In this case,
\[a^{\delta, N}_i(t) = \frac{\mbox{\# spikes in $(t-\delta,t]$ in pop. $i$}}{\delta N} \approx \frac{\lambda(u_i^N(t)) \delta N}{\delta N}=\lambda(u_i^N(t)).\]
They show that at fixed voltage $u$, the stationary distribution of the activity $a^{\delta,N}$ evolving according to (\ref{eq:jumpratesact}) is approximately Poisson with rate $\lambda(u)$. This corresponds to the regime considered in \cite{buicecowanchow}.

In the limit $N\rightarrow\infty$, $a_j(t) = \lambda(u_j(t))$ and the system reduces to the voltage-based network equation
\[\tau_m \dot{u}_j(t) = -u_j(t) + \sum_{j=1}^P w_{ij} \lambda(u_j(t)). \]

\subsubsection*{Case 3: $\delta=\Delta$, $\tau_m \gg \tau_a \rightarrow 0$ }

The third regime has not been considered explicitly in \cite{bressloffnewby}. It is the one which is relevant for us.

We go back to our original definition of the activity and fix the time window $\delta$ to be the length of an action potential $\Delta$.
We assume that the potential evolves slowly, $\tau_m \gg 0$. Speeding up time, we define
\[\tilde{u}_i^N(t) = u_i^N(t\tau_m).\]
Then
\begin{align*}
	\tilde{u}_i^N(t) 
	& = \sum_{j=1}^P w_{ij} \int_{-\infty}^{t\tau_m} \frac{1}{\tau_m} e^{-\frac{1}{\tau_m}(t\tau_m -s)} a_j^N(s) ds \\
	& =  \sum_{j=1}^P w_{ij} \int_{-\infty}^t e^{-(t-s)} a_j^N(s\tau_m) ds
\end{align*}
For some large $n$,
\begin{align*}
\tilde{u}_i^N(t) 
& \approx \sum_{j=1}^P w_{ij} \sum_{k=-\infty}^{[tn]-1} e^{-(t-\frac{k}{n})} \int_{\frac{k}{n}}^{\frac{k+1}{n}} a_j^N(s\tau_m) ds.
\end{align*}
\sloppy{The potentials $\tilde{u}_i^N$ therefore only depend on the time-averaged activities given for ${\frac{k}{n}\leq t < \frac{k+1}{n}}$ as
\[\tilde{a}_i^N(t) = n \int_{\frac{k}{n}}^{\frac{k+1}{n}} a_i^N(s\tau_m) ds.\]
We have
\begin{equation}
\label{eq:fsevoltage2}
\tilde{u}_i^N(t) 
 \approx \sum_{j=1}^P w_{ij}  \sum_{k=-\infty}^{[tn]-1}\frac{1}{n}  e^{-(t-\frac{k}{n})} \tilde{a}_j^N(\frac{k}{n}) \approx \sum_{j=1}^P w_{ij}  \int_{-\infty}^t  e^{-(t-s)} \tilde{a}_j^N(s) ds.
\end{equation}
If $\tau_a \ll \tau_m$, the activity relaxes to its stationary distribution quickly on this time scale. At fixed voltage $u$, under the stationary distribution $\nu(u)$, 
\[\tilde{a}_i^N(\frac{k}{n}) = n \int_{\frac{k}{n}}^{\frac{k+1}{n}} a_i^N(s\tau_m) ds \approx E_{\nu(u)}( a_i^N(\frac{k}{n} \tau_m) )= F(u),\]
with equality if $N\rightarrow\infty$. 
If $u$ is time-varying, then 
differentiation of (\ref{eq:fsevoltage2}) yields
\begin{equation}
\label{eq:actvoltagebased}
\begin{split}
	\frac{d}{dt} \tilde{a}_i(t) 
	& = \frac{d}{dt} F(\tilde{u}_i(t)) \\
	& = F'(\tilde{u}_i(t)) \Big(-\tilde{u}_i(t) + \sum_{j=1}^P w_{ij} F(\tilde{u}_j(t)) \Big)\\
	& = F'(F^{-1}(\tilde{a}_i(t))) \Big( - F^{-1}(\tilde{a}_i(t)) + \sum_{j=1}^P w_{ij} \tilde{a}_j(t) \Big).
\end{split}
\end{equation}

If $N <\infty$, then the finite size of the populations causes deviations from (\ref{eq:actvoltagebased}).
In order to determine these finite-size effects, in the next section we will set up a Markov chain $X^{P,N}$ to describe the evolution of the time-averaged activity $\tilde{a}^N$.

\section{A Markov Chain Model for the Activity }
\label{section:MC}

We describe the evolution of the time-averaged activity by a Markov chain $X^{P,N}$ with state space $E^{P,N} = \{0, \frac{1}{N}, \frac{2}{N}, \ldots, 1\}^P$.
We define the jump rates as
\begin{equation}
\label{eq:ourjumprates}
\begin{split}
	q^{P,N}(x, x+\frac{1}{N}e_i) & = N F'(F^{-1}(x_i)) \bigg(-F^{-1}(x_i)+\sum_{j=1}^P w_{ij} x_j\bigg)_+\\
	q^{P,N}(x, x-\frac{1}{N}e_i) & = N F'(F^{-1}(x_i)) \bigg(- F^{-1}(x_i)+\sum_{j=1}^P w_{ij} x_j\bigg)_-
\end{split}
\end{equation}
where for $x\in\R$, 
$x_+ := x\vee 0$, $x_- := -x \vee 0$, and where $e_i$ denotes the $i$-th unit vector.

The idea behind this choice is the following: the time-averaged activity tends to jump up (down) if the potential in the population, which is approximately given by $F^{-1}(x_i)$, is lower (higher) than the input from the other populations, which is given by $\sum_{j=1}^P w_{ij} x_j$. The probability that the activity jumps down (up) when the potential is lower (higher) than the input is assumed to be negligible. 
The jump rates are proportional to the difference between the two quantities, scaled by the factor $F'(F^{-1}(x_i))$. 
They are therefore higher in the sensitive regime where $F'\gg 1$, that is, where small changes in the potential have large effects on the activity.
If $a^N_i = F(u^N_i)$ in all populations $i$, then the system is in balance.

Note that the state space is naturally bounded since $\lim_{x \uparrow 1} F^{-1}(x) = \infty$ and \\ $\lim_{x \downarrow 0} F^{-1}(x) = - \infty$, such that $( -F^{-1}(x_k) + \sum_{l=1}^P w_{kl} x_l )_{+} = 0$ for $x$ with $x_k = 1-\frac{1}{N}$, $x_l \leq 1- \frac{1}{N}$, when $N\geq N_0$ is large enough, and similarly at $0$. 

We will see in Proposition \ref{prop:LLN} below that the Markov chain converges to the solution of (\ref{eq:actvoltagebased}) as the size of the populations $N$ goes to infinity.

In \cite{bressloffsystemsize} a different choice of jump rates was suggested in analogy to (\ref{eq:jumpratesact}):
\begin{align*}
	\tilde{q}(x, x+\frac{1}{N}e_i) & = N F'(F^{-1}(x_i)) \sum_{j=1}^P w_{ij} x_j\\
	\tilde{q}(x, x-\frac{1}{N}e_i) & = N F'(F^{-1}(x_i))  F^{-1}(x_i).
\end{align*}
Also this choice leads to (\ref{eq:actvoltagebased}) in the limit. In this picture, the jump rates are high in regions where the activity is high. Since, as explained above, one should think of the Markov chain as governing a slowly varying time-averaged activity, (\ref{eq:ourjumprates}) seems like a more natural choice.

The generator of $Q^{P,N}$ of $X^{P,N}$ is given for bounded measurable $f:E^{P,N} \rightarrow \R$ by
\begin{align*}
 Q^{P,N}f (x)  & = N \sum_{k=1}^P F'(F^{-1}(x_k)) \bigg( \Big(-F^{-1}(x_k) + \sum_{j=1}^P w_{kj} x_j\Big)_+ (f(x+\tfrac{1}{N}e_k)-f(x)) \\
& \qquad {} + \Big(-F^{-1}(x_k) + \sum_{j=1}^P w_{kj} x_j\Big)_- (f(x-\tfrac{1}{N}e_k) - f(x)) \bigg).
\end{align*}
Let $N_0$ be such that the jump rates out of the interval $\big[\frac{1}{N_0}, 1-\frac{1}{N_0}\big]$ are 0.
\begin{proposition}
\label{prop:LLN}
Let $X^P$ be the (deterministic) Feller process on $[1/N_0, 1-1/N_0]^P$ with generator
\[ L^Pf(x) = \sum_{k=1}^P F'(F^{-1}(x_k)) \Big( - F^{-1}(x_k) + \sum_{j=1}^P w_{kj} x_j \Big) \partial_kf(x).\]
If $X^{P,N}(0) \xrightarrow[N\rightarrow\infty]{d} X^P(0) $ , then $X^{P,N} \xrightarrow[N \rightarrow\infty]{d} X^P$ on the space of c\`adl\`ag functions \\ $D([0, \infty),[1/N_0,1-1/N_0]^P)$ with the Skorohod topology (where $\xrightarrow{d}$ denotes convergence in distribution).
\end{proposition}
\begin{proof}
By a standard theorem on the convergence of Feller processes (cf. \cite{kallenberg}, Thm. 19.25) it is enough to prove that for ${f \in \mathcal{C}^{\infty}([1/N_0,1-1/N_0]^P)}$ there exist bounded measurable $f_N$ such that ${\|f_N - f\|_{\infty} \xrightarrow{N\rightarrow\infty} 0}$ and ${\|Q^{P,N}f_N - L^Pf\|_{\infty} \xrightarrow{N \rightarrow\infty} 0}$.

Let thus ${f \in \mathcal{C}^{\infty}([1/N_0,1-1/N_0]^P)}$ and set ${f_N(x) = f\big(\tfrac{[x_1N]}{N}, ...,\tfrac{[x_PN]}{N}\big)}$. Then
it is easy to see that \[Q^{P,N}f_N(x) \xrightarrow{N\rightarrow\infty} L^Pf(x)\]
uniformly in $x$.
 \end{proof}

\section{Diffusion Approximation}
\label{section:diffusionapproximation}

\noindent 
We are now going to approximate $X^{P,N}$ by a diffusion process. 
To this end, we follow the standard approach due to Kurtz and derive a central limit theorem for the fluctuations of $X^{P,N}$. This will give us a candidate for a stochastic correction term to (\ref{eq:actvoltagebased}).

\subsection{A Central Limit Theorem}

We write
\[X_k^{P,N}(t) = X_k^{P,N}(0) + \int_0^t Q^{P,N} \pi_k (X^{P,N}(s)) ds + M_k^{P,N}(t)\]
where $\pi_k: (0,1)^P \rightarrow (0,1), x \mapsto x_k$, is the projection onto the $k$-th coordinate, and
\[ M_k^{P,N}(t) :=  X_k^{P,N}(t) - X_k^{P,N}(0) - \int_0^t Q^{P,N} \pi_k (X^{P,N}(s)) ds\]
is a martingale describing the fluctuations of the process. 

We start by determining the limit of these fluctuations.

\begin{proposition}
\label{prop:CLT}
\[ \left(\sqrt{N} M^{P,N}_k\right) \xrightarrow[N\rightarrow\infty]{d}  \bigg( \int_0^t  \Big(F'(F^{-1}(X_k^P(s))) \Big|-F^{-1}(X_k^P(s)) + \sum_{j=1}^P w_{kj}X^P_j(s) \Big| \Big)^{\frac{1}{2}} dB_k(s)\bigg)\]
on $\mathcal{D}([0,\infty), \R^P)$, where $B$ is a $P$-dimensional standard Brownian motion, and $X^P$ is the Feller process from Proposition \ref{prop:LLN} .
\end{proposition}
\begin{proof}
The bracket process of $M^{P,N}_k$ is given in terms of the carr\'e du champ operator as
\begin{align*}
\big\langle M^{P,N}_k \big\rangle_t
&= \int_0^t Q^{P,N} \pi_k^2(X^{P,N}(s)) - 2 Q^{P,N} \pi_k (X^{P,N}(s)) X_k^{P,N}(s) ds \\
& = \int_0^t \sum_{y\in E^{P,N}} q^{P,N}(X^{P,N}(s), y) (y_k - X_k^{P,N}(s))^2 ds \\
& = \frac{1}{N}  \int_0^t F'(F^{-1}(X^{P,N}_k(s))) \bigg( \Big(-F^{-1}(X^{P,N}_k(s)) + \sum_{j=1}^P w_{kj} X^{P,N}_j(s) \Big)_+ \\
& \qquad {} + \Big(-F^{-1}(X^{P,N}_k(s)) + \sum_{j=1}^P w_{kj} X^{P,N}_j(s)\Big)_- \bigg) ds\\
& = \frac{1}{N}  \int_0^t F'(F^{-1}(X^{P,N}_k(s))) \Big|-F^{-1}(X^{P,N}_k(s)) + \sum_{j=1}^P w_{kj} X^{P,N}_j(s)\Big| ds.
\end{align*}
Thus, \[ \big\langle\sqrt{N}M_k^{P,N}\big\rangle_t \xrightarrow{N\rightarrow\infty} \int_0^t F'(F^{-1}(X^{P}_k(s))) \Big|-F^{-1}(X^{P}_k(s)) + \sum_{j=1}^P w_{kj} X^{P}_j(s)\Big| ds \]
in probability.
For $k \neq l$,
\[ \big\langle M^{P,N}_k, M^{P,N}_l \big\rangle_t = \int_0^t \sum_y q^{P,N}(X^{P,N}(s), y) (y_k - X_k^{P,N}(s)) (y_l - X_l^{P,N}(s)) ds =0,\]
since for $y$ with $q^{P,N}(X^{P,N}(s),y) > 0$ at least one of $y_k - X_k^{P,N}(s)$ and $ y_l - X_l^{P,N}(s)$ is always $0$.
Now 
\[ E \left(\sup_{t} \sqrt{N} \|M^{P,N}(t)-M^{P,N}(t-)\|_2\right)  \leq \frac{1}{\sqrt{N}} \xrightarrow{N\rightarrow\infty}0,\]
and the statement follows by the martingale central limit theorem, see for example Theorem 1.4, Chapter 7 in \cite{ethierkurtz} .
\qquad \end{proof}

\noindent This suggests to approximate $X^{P,N}$ by the system of coupled diffusion processes
\begin{align*}
 da_k^{P,N}(t) 
&= F'(F^{-1}(a^{P,N}_k(t))) \Big( - F^{-1}(a^{P,N}_k(t)) + \sum_{j=1}^P w_{kj} a^{P,N}_j(t) \Big) dt\\
& \qquad {} + \frac{1}{\sqrt{N}} \Big(F'(F^{-1}(a^{P,N}_k(t))) \Big|-F^{-1}(a^{P,N}_k(t)) + \sum_{j=1}^P w_{kj} a^{P,N}_j(t)\Big|\Big)^{\frac{1}{2}} dB_k(t),
\end{align*}
$1 \leq k \leq P$.

Using It\^o's formula, 
we formally obtain an approximation for $u^{P,N}_k := F^{-1}(a_k^{P,N})$,
\begin{equation}
\label{eq:diffapprox}
\begin{split}
du^{P,N}_k(t) 
& = \Big( - u^{P,N}_k(t) + \sum_{j=1}^P w_{kj} F(u^{P,N}_j(t)) \\
& \qquad {} - \frac{1}{2N} \frac{F''(u^{P,N}_k(t))}{F'(u^{P,N}_k(t))^2} \Big|-u^{P,N}_k(t) + \sum_{j=1}^P w_{kj} F(u^{P,N}_j(t))\Big| \Big) dt \\
& \qquad {} + \frac{1}{\sqrt{N}}\Bigg(\frac{ \Big|-u^{P,N}_k(t) + \sum_{j=1}^P w_{kj} F(u^{P,N}_j(t))\Big|}{F'(u^{P,N}_k(t))}\Bigg)^{\frac{1}{2}}dB_k(t).
\end{split}
\end{equation}

\noindent Since the square root function is not Lipschitz continuous near 0, we cannot apply standard existence theorems to obtain a solution to (\ref{eq:diffapprox}) with the full multiplicative noise term.
Instead we will linearize around a deterministic solution to the neural field equation and approximate to a certain order of $\frac{1}{\sqrt{N}}$.

\subsection{Fluctuations around the Traveling Wave}
\label{section:fluctuationsaroundtw}

Let $\bar{u}$ be a solution to the neural field equation (\ref{eq:nfe}). To determine the finite-size effects on $\bar{u}$, 
we consider a spatially extended network, that is, we look at populations 
distributed over an interval $[-L,L] \subset \R$ and use the stochastic integral 
derived in Proposition \ref{prop:CLT} to describe the local fluctuations on this 
interval.

Let $m \in \mathbb{N}$ be the density of populations on $[-L,L]$ and consider $P=2mL$ populations located at $\frac{k}{m}, k \in \{-mL, -mL+1, ..., mL-1\}$. 
We choose the weights $w_{kl}$ as a discretization of the integral kernel $w:\R \rightarrow [0,\infty)$,
\begin{equation}
\label{eq:weights1}
w^m_{kl} = \int_{\frac{l}{m}}^{\frac{l+1}{m}} w(\tfrac{k}{m}- y) dy, \ \ -mL \leq k,l \leq mL-1.
\end{equation}
Since we think of the network as describing only a section of the actual domain $\R$, we add to each population an input $F(\bar{u}_t(-L))$ and $F(\bar{u}_t(L))$, respectively, at the boundaries with corresponding weights

\begin{equation}
\begin{aligned}
\label{eq:weights2}
	w^{m,+}_k & = \int_{L}^{\infty} w(\tfrac{k}{m} - y) dy,  \\
	w^{m,-}_k & = \int_{- \infty}^{-L} w(\tfrac{k}{m} - y) dy.
\end{aligned}
\end{equation}

\noindent Fix a population size $N \in \N$. Set $\bar{u}_k(t) = \bar{u}(\tfrac{k}{m}, t)$ and for $u\in\R^P$,
\[\hat{b}^m_k(t,u) = - u_k(t) + \sum_{l=-mL}^{mL-1} w^m_{kl} F(u_l(t)) + w_k^{m,+} F(\bar{u}(L,t)) + w_k^{m,-} F(\bar{u}(-L,t)).\]
We write 
\begin{equation}
\label{eq:decompu}
u_k = \bar{u}_k+v_k
\end{equation}
and assume that $v_k$ is of order $1/\sqrt{N}$. Linearizing (\ref{eq:diffapprox}) around $(\bar{u}_k)$ we obtain the approximation
\[du_k(t)  = \hat{b}_k^m(t,u) dt + \frac{1}{\sqrt{N F'(\bar{u}_k(t))}} |\hat{b}^m_k(t,\bar{u})|^{\frac{1}{2}} dB_k(t)\]
to order $1/\sqrt{N}$.

Note that $\hat{b}^m_k(t,\bar{u}) \approx \partial_t \bar{u}_k(t) = 0$ for a stationary solution $\bar{u}$, with equality if $\bar{u}$ is constant. The finite-size effects are hence of smaller order. 
Since the square root function is not differentiable at $0$ we cannot expand further.

However, the situation is different if we linearize around a moving pattern. We consider the traveling wave solution $u^{TW}_t(x) = \hat{u}(x-ct)$ to (\ref{eq:nfe}) and we assume without loss of generality that $c>0$. Then $\hat{b}^m_k(t,\bar{u}) \approx \partial_t \bar{u}_k(t) = - c \partial_x \tw_t <0$. This monotonicity property allows us to approximate to order $1/N$ in (\ref{eq:diffapprox}). Indeed, note that since $\hat{u}$ and $F$ are increasing,
\begin{equation}
\label{eq:positive}
\begin{split}
& - \hat{b}^m_k(t, \tw_t) \\
& = \tw_t(\tfrac{k}{m}) - \sum_l w^m_{kl} F(\tw_t(\tfrac{l}{m})) - w_k^{m,+} F(u^{TW}_t(L)) - w_k^{m,-} F(u^{TW}_t(-L))\\
&  \geq \tw_t(\tfrac{k}{m})  - \int_{-L}^{\infty} w(\tfrac{k}{m}-y) F(\tw_t)(y) dy - \int_{-\infty}^{-L} w(\tfrac{k}{m}-y) (F(u^{TW}_t(-L)) dy \\
&  =  c\hat{u}_x(\tfrac{k}{m}-ct)  - \int_{-\infty}^{-L} w(\tfrac{k}{m}-y) ( F(\tw_t(-L))-F(u^{TW}_t(y))) dy \\
& \geq  c\hat{u}_x(\tfrac{k}{m}-ct)  - \left( F(\tw_t(-L)) - F(a_1)\right) \\
& \xrightarrow{L\rightarrow \infty} c\hat{u}_x(\tfrac{k}{m}-ct) > 0.
\end{split}
\end{equation}
So for $L$ large enough, $-\hat{b}^m_k(t, \tw_t)>0$ 
and we have, using Taylor's formula and (\ref{eq:decompu}), 
\begin{align*}
\bigg(\frac{|\hat{b}^m_k(t,u)|}{F'(u_k(t)))}\bigg)^{\frac{1}{2}} 
& = \bigg(\frac{-\hat{b}^m_k(t, \tw_t)}{F'(\tw_t(\tfrac{k}{m}))}\bigg)^{\frac{1}{2}} 
 + \frac{1}{2\sqrt{-\hat{b}^m_k(t, \tw_t) F'(\tw_t(\tfrac{k}{m}))}} \\ 
& \qquad 
	\bigg( \frac{F''(\tw_t(\tfrac{k}{m}))}{F'(\tw_t(\tfrac{k}{m}))} \hat{b}^m_k(t, \tw_t) v_k(t) \\
	& \qquad {} + v_k(t) - \sum_l w_{kl} F'(\tw_t(\tfrac{l}{m})) v_l(t) \bigg) + O\Big(\frac{1}{N}\Big).
\end{align*}
As a possible diffusion approximation in the case of traveling wave solutions we therefore obtain the system of stochastic differential equations
\begin{equation}
\label{eq:systemofsde}
\begin{split}
du_k(t) & = \Big(\hat{b}^m_k(t, u) + \frac{1}{2N} \frac{F''(\tw_t(\tfrac{k}{m}))}{F'(\tw_t(\tfrac{k}{m}))^2} \hat{b}^m_k(t, \tw_t)\Big) dt\\
&  \qquad {} + \frac{1}{\sqrt{N}}\bigg[ \bigg( \frac{-\hat{b}^m_k(t, \tw_t)}{F'(\tw_t(\tfrac{k}{m}))}\bigg)^{\frac{1}{2}} + \frac{1}{2\sqrt{-\hat{b}^m_k(t, \tw_t) F'(\tw_t(\tfrac{k}{m}))}} \\
	& \qquad \bigg( \frac{F''(\tw_t(\tfrac{k}{m}))}{F'(\tw_t(\tfrac{k}{m}))} \hat{b}^m_k(t, \tw_t) v_k(t) \\
	& \qquad {} + v_k(t) - \sum_l w_{kl} F'(\tw_t(\tfrac{l}{m})) v_l(t) \bigg) \bigg] dB_k(t),
\end{split}
\end{equation}
for which there exists a unique solution as we will see in the next section.

\section{The Continuum Limit}
\label{section:continuumlimit}

\noindent In this section we take the continuum limit of the network of diffusions (\ref{eq:systemofsde}), that is, we let the size of the domain and the density of populations go to infinity in order to obtain a stochastic neural field equation with a noise term describing the fluctuations around the deterministic traveling wave solution due to finite size effects.

We thus have to deal with functions that `look almost like the wave' and choose to work in the space 
$\mathcal{S}:= \left\{ u:\R \rightarrow \R: \ u - \twp \in L^2 \right\}$.
Note that since for $u_1, u_2 \in\mathcal{S}$, 
$\|u_1-u_2\| < \infty$, the $L^2$-norm induces a topology on $\mathcal{S}$.

\subsection{A Word on Correlations}
\label{subsection:correlations}

Recall the definition of the Markov chain introduced in section \ref{section:MC}.
Note that as long as we allow only single jumps in the evolution, meaning that there will not be any jumps in the activity in two populations at the same time, the martingales associated with any two populations will be uncorrelated, yielding independent driving Brownian motions in the diffusion limit (cf. Proposition \ref{prop:CLT}). 

This only makes sense for populations that are clearly distinguishable. In order to determine the fluctuations around traveling wave solutions, we consider spatially extended networks of populations. The population located at $x \in \R$ is to be understood as the ensemble of all neurons in the $\epsilon$-neighborhood $(x-\epsilon, x+\epsilon)$ of $x$ for some $\epsilon > 0$.
If we consider two populations located at $x,y \in \R$ with $|x-y| < 2\epsilon$, then they will overlap. Consequently, simultaneous jumps will occur, leading to correlations between the driving Brownian motions.

Thus the Markov chain model (and the associated diffusion approximation) is only 
appropriate as long as the distance between the individual populations is large 
enough. When taking the continuum limit, we therefore adapt the model by 
introducing correlations between the driving Brownian motions of populations 
lying close together.

\subsection{The Stochastic Neural Field Equation}
We start by defining the limiting object.
For $u \in \mathcal{S}$ and $t \in [0,T]$ set 
\[ b(t,u)(x) = - u(x) + \int_{-\infty}^{\infty} w(x-y) F(u(y)) dy  = -u(x) + w \ast F(u)(x) . \]

\noindent Let $\mathcal{W}^Q$ be a (cylindrical) $Q$-Wiener process on $L^2$ with covariance operator $\sqrt{Q}$ given as
$\sqrt{Q}h(x) = \int_{-\infty}^{\infty} q(x,y) h(y) dy$ for some symmetric kernel $q(x,y)$ with $q(x,\cdot) \in L^2 \cap L^1$ for all $x \in \R$ and $\sup_{x \in \R } (\|q(x,\cdot)\| + \|q(x,\cdot)\|_1) < \infty$. (Details on the theory of $Q$-Wiener processes can be found in \cite{prevotroeckner, daprato}.) We assume that the dispersion coefficient is given as the multiplication operator associated with $\sigma: [0,T] \times \mathcal{S} \rightarrow L^2(\R)$, which we also denote by $\sigma$, where $\sigma$ is Lipschitz continuous with respect to the second variable uniformly in $t\leq T$, that is,  we assume that there exists $L_{\sigma}>0$ such that for all $u_1, u_2 \in \mathcal{S}$ and $t \in [0,T]$,
\begin{equation}
\label{eq:Lipschitzsigma}
\|\sigma(t,u_1)-\sigma(t,u_2)\| \leq  L_{\sigma} \|u_1-u_2\|.
\end{equation}

\noindent The correlations are described by the kernel $q$. For $f,g$ in $L^2(\R)$,
\[E(\langle f, \mathcal{W}^Q_t \rangle \langle g, \mathcal{W}^Q_t \rangle) = \int f(x) Qg(x) dx = \int f(x) \int q(x,z) \int q(z,y) g(y) dy dz dx,\]
so formally,
\["E(\mathcal{W}^Q_t(x) \mathcal{W}^Q_t(y)) =  E(\langle \delta_x, \mathcal{W}^Q_t \rangle \langle \delta_y, \mathcal{W}^Q_t \rangle) = q \ast q(x,y)", \]
where we denote by $q\ast q(x,y) $ the integral $\int q(x,z) q(z,y) dz$.
We could for example take
\begin{equation}
\label{eq:q} q(x,y) = q(x-y) = \frac{1}{2\epsilon}\mathds{1}_{(-\epsilon, \epsilon)} (x-y)
\end{equation}
for some small $\epsilon > 0$ (cf. section \ref{subsection:correlations}).

We have $\sigma(t,v) \in L_2^0$ since by Parseval's identity
	\begin{align*}
	\|\sigma(t,v)\|_{L_2^0}^2 
	& = \sum_k \|\bar{\sigma}(t,v) Q^{\frac{1}{2}} e_k\|_2^2 = \int \sigma(t,v)^2(x) \|q(x,\cdot)\|^2 dx \\
	& \leq \sup_x \|q(x,\cdot)\|^2 \|\sigma(t,v)\|^2 < \infty.
	\end{align*}
Note that for uncorrelated noise (i.e. Q=E), this is not the case. Therefore, in \cite{riedler} Riedler and Buckwar derive the central limit theorem in the Sobolev space $H^{-\alpha}$. Splitting up the limiting procedures, $N \rightarrow \infty$ and continuum limit, allows us to incorporate correlations and finally to work in the more natural function space $L^2$.
\begin{proposition}
\label{prop:existencesnfe}
For any initial condition $u^0 \in \mathcal{S}$, the stochastic evolution equation
\begin{equation}
\label{eq:fsesnfe}
\begin{split}
du_t( x) & = \Big( -u_t + w \ast F(u_t) +  \frac{1}{2N} \frac{F''(\tw_t)}{F'(\tw_t)^2} \partial_t \tw_t \Big) \ dt  + \sigma(t, u_t) d\mathcal{W}^Q_t(x) \\
u_0 & = u^0,
\end{split}
\end{equation}
has a unique strong $\mathcal{S}$-valued solution.
$u$ has a continuous modification.
For any $p \geq 2$, \[ E \Big(\sup_{t\in [0,T]} \|u_t-\tw_t\|^p\Big) < \infty. \]
\end{proposition}

\noindent For a proof see for example Prop.~6.5.1 in \cite{langthesis}.

\subsection{Embedding of the Diffusion Processes}

As a next step we embed the systems of coupled diffusion processes (\ref{eq:systemofsde}) into $L^2(\R)$.
Let $m \in \mathbb{N}$ be the population density and $L^m \in \mathbb{N}$ the length of the domain with $L^m \uparrow \infty$ as $m \rightarrow \infty$.
For $k \in \{-mL^m, -mL^m+1, ...,  mL^m-1\}$ set $I^m_k = [\frac{k}{m}, \frac{k+1}{m})$ and $J^m_k = (\frac{k}{m}-\frac{1}{4m}, \frac{k}{m}+\frac{1}{4m})$, and let

\[W^m_k(t) = 2m \langle \mathcal{W}^Q_t , \mathds{1}_{J^m_k} \rangle\]
be the average of $\mathcal{W}^Q_t$ on the interval $J^m_k$. 
Then the $W^m_k$ are one-dimensional Brownian motions with covariances
\begin{align*}
	E(W^m_k W^m_l)
	& = 4m^2 \langle \sqrt{Q} \mathds{1}_{J^m_k}, \sqrt{Q} \mathds{1}_{J^m_l} \rangle 
	= 4m^2 \int_{J^m_k} \int_{J^m_l} q \ast q(y,z) dy dz.
\end{align*}
Note that the Brownian motions are independent as long as 
$m < \frac{1}{4\epsilon}$.

For $m \in \mathbb{N}$ let $\hat{\sigma}^m:[0,T] \times \R^P \rightarrow \R^P$ and assume that there exists $L_{\hat{\sigma}^m} > 0$ such that for any $t \in [0,T]$ and $u_1, u_2 \in \R^P$,
\[\|\hat{\sigma}^m(t,u_1) - \hat{\sigma}^m(t,u_2)\|_2 \leq L_{\hat{\sigma}^m} \|u_1-u_2\|_2.\] 
Consider the system of coupled stochastic differential equations
\begin{align*}
	du^m_k(t) 
	& = \hat{b}^m_k(t, (u^m_k)) + \frac{1}{2N} \frac{F''(\tw_t(\tfrac{k}{m}))}{F'(\tw_t(\tfrac{k}{m}))^2} \hat{b}^m_k(t, (\tw_t(\tfrac{k}{m}))_k) \\
	& \qquad {} + w^{m,+}_k F(\tw_t(L^m)) + w^{m,-}_k F(\tw_t(-L^m)) \Big) dt \\
	& \qquad {} + \hat{\sigma}_k^m(t,u^m(t)) dW_k^m(t), \qquad \qquad -mL^m \leq k \leq mL^m-1,
\end{align*}
with weights as in (\ref{eq:weights1}) and (\ref{eq:weights2}).

We identify $u = (u_k)_{-mL^m \leq k \leq mL^m-1} \in \mathbb{R}^P$ with its piecewise constant interpolation as an element of $L^2$ via the embedding
\[\iota^m(u) = \sum_{k=-mL^m}^{mL^m-1} u_k \mathds{1}_{I^m_k}.\]
For $u \in \mathcal{C}(\mathbb{R})$ set
\begin{align*}
	\pi^{m}(u) & = \sum_{k=-mL^m}^{mL^m-1} u(\tfrac{k}{m}) \mathds{1}_{I^m_k}. \\
\end{align*}
Then $u^m_t := \iota^m((u^m_k(t))_k)$ satisfies
\begin{equation}
\label{eq:systemofdiff}
\begin{split}
du^m_t(x) & =  b^m(t, u^m_t)(x) dt + \frac{1}{2N} \pi^m\Big(\frac{F''(\tw_t)}{F'(\tw_t)^2}\Big) b^m(t, \pi^m(\tw_t)) dt \\
& \qquad  + \sigma^m(t,u^m_t) \circ \Phi^m d\mathcal{W}^Q_t(x),
\end{split}
\end{equation}
where $b^m:[0,T]\times L^2(\R) \rightarrow L^2(\R)$ and $\Phi^m:L^2(\R) \rightarrow L^2(\R)$ are given as
\begin{align*} 
	b^m(t,u) & =  - u_t + \sum_k \bigg[\int_{-L^m}^{L^m} w(\tfrac{k}{m}-y) F(u_t)(y) dy \\
	& \qquad + w^{m,+}_k F(\tw_t(L^m)) + w^{m,-}_k F(\tw_t(-L^m))
	\bigg] \mathds{1}_{I^m_k}, \\
	\Phi^m (u) & = 2m \sum_{k=-mL^m}^{mL^m-1} \langle u, \mathds{1}_{J^m_k} \rangle  \mathds{1}_{I^m_k},
\end{align*}
and where $\sigma^m:[0,T] \times L^2(\R) \rightarrow L^2(\R)$ is such that for $u\in\R^P$, $\sigma^m(t,\iota^m(u))= \hat{\sigma}^m_k(t,u)$ on $I^m_k$.
We assume joint continuity and Lipschitz continuity in the second variable uniformly in $m$ and $t\leq T$, that is, there exists $L_{\sigma}>0$ such that
for $u_1,u_2 \in L^2(\R)$ and $t\leq T$,
\[\|\sigma^m(t,u_1)-\sigma^m(t,u_2)\| \leq L_{\sigma} \|u_1-u_2\|.\]
\begin{proposition}
For any initial condition $u^0 \in L^2(\R)$ there exists a unique strong $L^2$-valued solution $u^m$ to (\ref{eq:systemofdiff}).
$u^m$ admits a continuous modification. For any $p\geq 2$, 
\[E\Big(\sup_{t\leq T} \|u^m_t\|^p\Big) < \infty.\]
\end{proposition}
\begin{proof}
Again we check that the drift and diffusion coefficients are Lipschitz continuous.
Note that 
\begin{equation}
\label{eq:riemannsum}
\begin{split}
\sum_k \frac{1}{m} w(\tfrac{k}{m}, y)  & =  \int_{-L^m}^{L^m} w(x-y) dx + \sum_k \int_{\frac{k}{m}}^{\frac{k+1}{m}} w(\tfrac{k}{m}-y)-w(x-y) dx \\
& \leq \|w\|_1 + \sum_k \int_{\frac{k}{m}}^{\frac{k+1}{m}} \int_{\frac{k}{m}}^{\frac{k+1}{m}} |w_x(z-y)| dz dx \\
& = 1 + \sum_k \frac{1}{m} \int_{\frac{k}{m}}^{\frac{k+1}{m}} |w_x(z-y)| dz \leq 1 + \frac{1}{m}\|w_x\|_1.
\end{split}
\end{equation}
Therefore, for $u_1, u_2 \in L^2(\R)$,
\begin{align*}
	 \|b^m(t, u_1) - b^m(t, u_2)\|_2^2 &
	 \leq 2 \|u_1-u_2\|^2 + 2 \|F'\|_{\infty}^2 \int_{-L^m}^{L^m} \sum_k \frac{1}{m}w(\tfrac{k}{m}-y) \left( u_1(y) -u_2(y) \right)^2 dy \\
	& \leq 2 \|u_1-u_2\|^2 + 2 \|F'\|_{\infty}^2 \Big( 1 + \frac{1}{m}\|w_x\|_1\Big) \|u_1-u_2\|^2
\end{align*}
and for an orthonormal basis $(e_k)$ of $L^2(\R)$ we obtain, using Parseval's identity,
\begin{align*}
{} & \|\left(\sigma^m(t,u_1) - \sigma^m(t, u_2)\right)\circ \Phi^m\|_{L^0_2}^2 \\
& = \sum_k \Big\|\left( \sigma^m(t, u_1)-\sigma^m(t,u_2)\right) \sum_l 2m \langle e_k, \sqrt{Q} \mathds{1}_{J^m_l} \rangle \mathds{1}_{I^m_l} \Big\|^2 \\
& = \int \left( \sigma^m(t, u_1)-\sigma^m(t,u_2)\right)^2(x) \sum_l  4m^2 \int \Big( \int_{J^m_l}q(z,y) dy \Big)^2 dz \ \mathds{1}_{I^m_l}(x) dx \\
& \leq \int \left( \sigma^m(t, u_1)-\sigma^m(t,u_2)\right)^2(x) \sum_l 2m \int \int_{J^m_l} q(z,y)^2 dy  dz \ \mathds{1}_{I^m_l}(x) dx\\
& = \int \left( \sigma^m(t, u_1)-\sigma^m(t,u_2)\right)^2(x) \sum_l 2m \int_{J^m_l} \|q(y,\cdot)\|^2 dy \ \mathds{1}_{I^m_l}(x) dx\\
& \leq \sup_x \|q(x,\cdot)\|^2 L_{\sigma}^2 \|u_1-u_2\|^2  \qedhere
\end{align*}
\end{proof}

\subsection{Convergence}
We are now able to state the main convergence result.
We will need the following assumption on the kernel $w$.
\begin{assumption}
There exists $C_w>0$ such that for $x\geq 0$,
\begin{equation}
\label{eq:assonw} 
\int_x^{\infty} w(y) dy \leq C_w w(x).
\end{equation}
\end{assumption}
\sloppy{That assumption is satisfied for classical choices of $w$ such as ${w(x) = \frac{1}{2\sigma}e^{-\frac{|x|}{\sigma}}}$ or ${w(x) = \frac{1}{\sqrt{2\pi\sigma^2}} e^{-\frac{x^2}{2\sigma^2}}}$.

\begin{theorem}
\label{thm:mainthm}
Fix $T>0$. Let $u$ and $u^m$ be the solutions to (\ref{eq:fsesnfe}) and (\ref{eq:systemofdiff}), respectively.
Assume that 
\begin{enumerate}[(i)]
	\item $\sup_k \sup_{x \in I^m_k} \|2m \sqrt{Q}\mathds{1}_{J^m_k}-q(x,\cdot)\| \xrightarrow{m \rightarrow \infty} 0$,
	\item for any $u: [0,T] \rightarrow \mathcal{S}$ with $sup_{t\leq T} \|u_t - \twp\| < \infty$, 
	\[\sup_{t\leq T} \|\sigma^m(t,u_t\mathds{1}_{(-L^m, L^m)})-\sigma(t,u_t)\| \xrightarrow{m \rightarrow \infty} 0.\]
\end{enumerate}
\sloppy{Then for any initial conditions ${u_0^m\in L^2(\R)}$, ${u_0 \in \mathcal{S}}$ such that ${\|u_0^m-u_0\|_{L^2((-L^m, L^m))} \xrightarrow{m\rightarrow \infty} 0}$, and for all $p\geq2$,}
\[  \mathds{E} \Big( \sup_{t \in [0,T]} \|u^m_t-u_t\|_{L^2((-L^m, L^m))}^p \Big) \xrightarrow{m \rightarrow \infty} 0.\]
\end{theorem}
We postpone the proof to section \ref{section:proofs}.

\begin{remark}
Let $\epsilon > 0$. The kernel $q(x,y) = \frac{1}{2\epsilon}\mathds{1}_{(x-\epsilon, x+\epsilon)}(y)$ satisfies assumption (i) of the theorem. Indeed, 
note that for $x,z$ with $|x-z| \leq \frac{1}{m}$, \\ $|\{y: \mathds{1}_{(x-\epsilon,x+\epsilon)}(y) \neq \mathds{1}_{(z-\epsilon, z+\epsilon)}(y)\}| \leq |z-x| \leq \frac{1}{m}$. Therefore we obtain that
for all $k$ 
and for any $x \in I^m_k$,
\begin{align*}
	& \|2m \sqrt{Q}\mathds{1}_{J^m_k}-q(x,\cdot)\|_2^2  = 4m^2 \int_{-\infty}^{\infty} \Big( \int_{J^m_k} q(z,y) - q(x,y) dz \Big)^2 dy \\
	& \leq 2m \int_{-\infty}^{\infty} \int_{J^m_k} \left( q(z,y) - q(x,y)  \right)^2 dz \ dy \\
	& \leq 2 \int_{J^m_k}  \frac{1}{4\epsilon^2}  dz 
	\leq \frac{1}{4 \epsilon^2 m} \xrightarrow{m\rightarrow \infty} 0.
\end{align*}
\end{remark}

The theorem applies to the case of the fluctuations described in section \ref{section:fluctuationsaroundtw}. In order to ensure that the diffusion coefficients are in $L^2(\R)$, we cut off the noise outside a compact set $\{ \partial_x \tw_t \geq \delta\}, \delta>0$. Note that the neglected region moves with the wave such that we always retain the fluctuations in the relevant regime away from the fixed points.

\begin{theorem}
\label{thm:diffcoeff}
Assume that the wave speed is strictly positive, $c>0$. Fix $\delta>0$.
The diffusion coefficients as derived in Section \ref{section:fluctuationsaroundtw},
\begin{align*}
	\sigma(t,u) &= \frac{1}{\sqrt{N}}\Big( \alpha(t) + \beta(t) (u - \tw_t) \\
	& \qquad {} - \gamma(t) w\ast (F'(\tw_t) (u-\tw_t)) \Big) \mathds{1}_{\{\partial_x \tw_t \geq \delta\}}, \\
	\sigma^m(t,u) &=  \frac{1}{\sqrt{N}} \Big( \alpha^m(t) + \beta^m(t) (u-\pi^m(\tw_t))\\
	& \qquad {} - \gamma^m(t) \pi^m(w \ast (\pi^m(F'(\tw_t)) (u-\pi^m(\tw_t))))\Big) \mathds{1}_{\{\partial_x \tw_t \geq \delta\}}, 
\end{align*}
where
\begin{align*}
	\alpha(t) &= \sqrt{\frac{ |- \tw_t + w\ast F(\tw_t)|}{F'(\tw_t)}} = \sqrt{\frac{ c \partial_x \tw_t}{F'(\tw_t)}}, \\
	\beta(t) &= \frac{1}{2 \sqrt{c \partial_x \tw_t F'(\tw_t)}} \left( -\frac{F''(\tw_t(x))}{F'(\tw_t(x))} c\partial_x \tw_t + 1 \right), \\
	\gamma(t) &= \frac{1}{2 \sqrt{c \partial_x \tw_tF'(\tw_t)}},\\
	\alpha^m(t) &= \sqrt{\frac{-b^m(t,\pi^m(\tw_t))}{\pi^m(F'(\tw_t))}} \mathds{1}_{[-L^m, L^m)}, \\
	\beta^m(t) &= \frac{1}{2 \sqrt{-b^m(t,\pi^m(\tw_t)) \pi^m(F'(\tw_t))}} \\
	& \qquad {} \times \left(  -\frac{\pi^m(F''(\tw_t))}{\pi^m(F'(\tw_t))} (-b^m(t,\pi^m(\tw_t)))  + 1 \right) \mathds{1}_{[-L^m, L^m)},\\
	\gamma^m(t) &= \frac{1}{2 \sqrt{-b^m(t,\pi^m(\tw_t)) \pi^m(F'(\tw_t))}} \mathds{1}_{[-L^m, L^m)},
\end{align*}
are jointly continuous and Lipschitz continuous in the second variable with Lipschitz constant uniform in $m$ and $t\leq T$, and satisfy condition (ii) of Theorem \ref{thm:mainthm}.
\end{theorem}

\noindent For a proof see Thm.~6.5.5 in \cite{langthesis}.


\section{Proof of Theorem \ref{thm:mainthm}}
\label{section:proofs}

Set $v_t = u_t - \tw_t$ and $v_t^m = u^m_t - \pi^m(\tw_t)$.
Note that
\begin{equation}
\label{eq:convpim}
\begin{split}
	\int_{-L^m}^{L^m} & (\pi^m(\tw_t)(x)-\tw_t(x))^2 dx  
	 \leq \sum_k \int_{\frac{k}{m}}^{\frac{k+1}{m}} \Big( \int_{\frac{k}{m}}^{\frac{k+1}{m}} \partial_x \tw_t(z) dz\Big)^2 dx \\
	&	\leq \frac{1}{m^2} \sum_k \int_{\frac{k}{m}}^{\frac{k+1}{m}} \left( \partial_x \tw_t(z)\right)^2 dz \leq \frac{1}{m^2} \|\twp_x\|^2.
\end{split}
\end{equation}
For the proof of the theorem it therefore suffices to show that 
\[E \Big(\sup_{t\leq T} \|v_t-v^m_t\|_2^p \Big) \xrightarrow{m\rightarrow\infty} 0,\]
since this will imply that         
\begin{align*}
& E \Big( \sup_{t\leq T} \|u_t - u^m_t\|_{L^2((-L^m, L^m))}^p \Big) \\
& \leq const \times \bigg[E \Big(\sup_{t\leq T} \|\tw_t-\pi^m(\tw_t)\|_{L^2((-L^m, L^m))}^p \Big)  + E \Big( \sup_{t\leq T} \|v_t-v^m_t\|_2^p \Big) \bigg]\xrightarrow{m \rightarrow \infty} 0.
\end{align*}


By It\^o's formula,
\begin{align*}
	& \frac{1}{2} d\|v^m_t-v_t\|_2^2 \\
	& = \langle b^m(t, v^m_t+\pi^m(\tw_t)) - b(t, v_t+\tw_t) -  \pi^m (\partial_t\tw_t) + \partial_t \tw_t, v^m_t-v_t \rangle dt \\
	& \qquad {} + \frac{1}{2N} \Big\langle \pi^m\Big(\frac{F''(\tw_t)}{F'(\tw_t)^2}\Big) b^m(t, \pi^m(\tw_t)) - \frac{F''(\tw_t)}{F'(\tw_t)^2} \partial_t \tw_t, v^m_t-v_t \Big\rangle dt \\
	& \qquad {} + \frac{1}{2} \|\sigma^m(t,v^m_t+ \pi^m(\tw_t)) \circ \Phi^m - \sigma(t,v_t+\tw_t)\|_{L^0_2}^2\, dt   \\
	& \qquad {} + \langle v^m_t-v_t, \left( \sigma^m(t, v^m_t+\pi^m(\tw_t)) \circ \Phi^m - \sigma(t, v_t+\tw_t)\right) d\mathcal{W}^Q_t \rangle.
\end{align*}
In order to finally apply Gronwall's Lemma, we estimate the terms one by one.


\subsection{The Drift}
We start by regrouping the terms in a suitable way.
We have
\begin{align*}
{} & \|v^m_t - v_t + b^m(t, v^m_t+\pi^m(\tw_t)) - b(t, v_t+\tw_t) - \pi^m(\partial_t\tw_t) + \partial_t \tw_t\|^2 \\
   & = \int_{-\infty}^{\infty} \bigg[\sum_k \bigg( \int_{-L^m}^{L^m} w(\tfrac{k}{m}-y) F(v^m_t(y) + \pi^m(\tw_t)(y)) dy \\
	 & \qquad {} + \int_{L^m}^{\infty} w(\tfrac{k}{m}- y) F(\tw_t(L^m)) dy + \int_{-\infty}^{-L^m} w(\tfrac{k}{m}-y) F(\tw_t(-L^m)) dy \bigg) \mathds{1}_{I^m_k}(x) \\
	 & \qquad {} - \int_{-\infty}^{\infty} w(x-y) F(v_t(y)+\tw_t(y)) dy - \pi^m( w \ast  F(\tw_t) ) + w \ast F(\tw_t)(x) \bigg]^2 dx \\
   & \leq 6 \int_{-\infty}^{\infty} \bigg[ \sum_k  \int_{-L^m}^{L^m} w(\tfrac{k}{m}- y) \Big( F(v^m_t(y) + \pi^m(\tw_t)(y))- F(v_t(y)+\tw_t(y)) \Big) dy \mathds{1}_{I^m_k}(x) \bigg]^2 \\
	 & \qquad {} + \bigg[ \sum_k \int_{-L^m}^{L^m} \left( w(\tfrac{k}{m}- y) - w(x-y) \right) \left( F(v_t(y)+\tw_t(y))-F(\tw_t(y)) \right) dy \mathds{1}_{I^m_k}(x)\bigg]^2 \\
	 & \qquad {} + \left[ \sum_k \int_{L^m}^{\infty} w(\tfrac{k}{m}- y) \left( F(\tw_t(L^m)) - F(\tw_t(y)) \right) dy \mathds{1}_{I^m_k}(x) \right]^2 \\
	 & \qquad {} + \left[ \sum_k \int_{-\infty}^{-L^m} w(\tfrac{k}{m}- y) \left( F(\tw_t(-L^m)) - F(\tw_t(y)) \right) dy \mathds{1}_{I^m_k}(x) \right]^2 \\
	 & \qquad {} + \left[ \left( \int_{L^m}^{\infty} + \int_{-\infty}^{-L^m} \right) w(x-y) \left( F(v_t(y) + \tw_t(y)) - F(\tw_t(y)) \right) dy \right]^2\\
	 & \qquad {} + \bigg[ \int_{-L^m}^{L^m} w(x-y) \left( F(v_t(y) + \tw_t(y)) - F(\tw_t(y)) \right) dy  \mathds{1}_{(-\infty, -L^m) \cup \left[L^m, \infty\right)}(x) \bigg]^2 dx\\
   & =: 6 (S_1 + S_2 + S_3 + S_4 + S_5 + S_6).
\end{align*}
Using the Cauchy-Schwarz inequality we get
\begin{align*}
S_1	
	& \leq \|F'\|_{\infty}^2 \int_{-L^m}^{L^m} \sum_k \frac{1}{m}  w(\tfrac{k}{m}- y) \left( v^m_t(y) + \pi^m(\tw_t)(y)- v_t(y)-\tw_t(y) \right)^2 dy\\
	& \stackrel{(\ref{eq:riemannsum})}{\leq} 2 \Big(1 + \frac{1}{m}\|w_x\|_1\Big) \|F'\|_{\infty}^2 \Big( \|v^m_t-v_t\|_2^2 + \int_{-L^m}^{L^m} \big(\pi^m(\tw_t)(y)-\tw_t(y)\big)^2 dy \Big).
\end{align*}
With (\ref{eq:convpim}) it follows that
\[S_1 \leq 2 \Big(1 + \frac{1}{m}\|w_x\|_1\Big) \|F'\|_{\infty}^2 \Big( \|v^m_t-v_t\|_2^2 +  \frac{1}{m^2} \|\twp_x\|_2^2 \Big).\]
Another application of the Cauchy-Schwarz inequality yields
\begin{align*}
	S_2 
	& = \sum_k \int_{\frac{k}{m}}^{\frac{k+1}{m}} \bigg( \int_{-L^m}^{L^m} \left( w(\tfrac{k}{m}- y) - w(x-y) \right) \\
	& \qquad \left( F(v_t(y)+\tw_t(y))-F(\tw_t(y)) \right) dy \bigg)^2 dx \\
	& \leq \sum_k \frac{1}{m} \bigg( \int_{-L^m}^{L^m} \int_{\frac{k}{m}}^{\frac{k+1}{m}} |w_x(z-y)| dz \left( F(v_t(y)+\tw_t(y))-F(\tw_t(y)) \right) dy \bigg)^2 \\
	& \leq \sum_k \frac{1}{m} \int_{-L^m}^{L^m} \int_{\frac{k}{m}}^{\frac{k+1}{m}} |w_x(z-y)| dz dy \\
	& \qquad {} \times \int_{-L^m}^{L^m} \int_{\frac{k}{m}}^{\frac{k+1}{m}} |w_x(z-y)| dz \left( F(v_t(y)+\tw_t(y))-F(\tw_t(y)) \right)^2 dy\\
	& \leq \frac{1}{m^2} \|w_x\|_1^2 \|F'\|_{\infty}^2 \|v_t\|_2^2	
\end{align*}
Using integration by parts, (\ref{eq:riemannsum}), and assumption (\ref{eq:assonw}), we obtain
\begin{align*}
	S_3
	& = \sum_k  \frac{1}{m} \bigg( \underbrace{\left[ - \int_y^{\infty} w(z-\tfrac{k}{m}) dz \left( F(\tw_t(L^m)) - F(\tw_t(y)) \right) \right]_{y=L^m}^{\infty}}_{=0}\\
	& \qquad {} - \int_{L^m}^{\infty} \int_y^{\infty} w(z-\tfrac{k}{m}) dz \  F'(\tw_t(y))\partial_x\tw_t(y) dy \bigg)^2 \\
	& \leq C_w^2 \sum_k \frac{1}{m} \left( \int_{L^m}^{\infty} w(y-\tfrac{k}{m}) F'(\tw_t(y)) \partial_x \tw_t(y) dy \right)^2 \\
	& \leq C_w^2 \Big(1 + \frac{1}{m}\|w_x\|_1\Big) \|F'\|_{\infty}^2 \int_{L^m}^{\infty} (\partial_x\tw_t(y))^2 dy.
\end{align*}
Analogously,
\[	S_4 	 \leq C_w^2 \Big(1 + \frac{1}{m}\|w_x\|_1\Big) \|F'\|_{\infty}^2 \int_{-\infty}^{-L^m} (\partial_x \tw_t(y))^2 dy. \]
Last we observe that
\begin{align*}
	S_5 
	& \leq \|F'\|_{\infty}^2 \bigg( \int_{L^m}^{\infty} + \int_{-\infty}^{-L^m} \bigg) v_t^2(y) dy.	
\end{align*}
and
\begin{align*}
	S_6 
	& \leq \|F'\|_{\infty}^2 \bigg( \int_{L^m}^{\infty} + \int_{-\infty}^{-L^m} \bigg)  \bigg( \int_{-L^m}^{L^m} w(x-y) |v_t(y)| dy \bigg)^2 dx \\
	& \leq \|F'\|_{\infty}^2 \bigg( \int_{L^m}^{\infty} + \int_{-\infty}^{-L^m} \bigg)  ( w \ast |v_t| (x) )^2 dx.
\end{align*}
Finally we consider
\begin{align*}
	S_7 
	&:= \Big\|\frac{F''(\tw_t)}{F'(\tw_t)^2} \partial_t \tw_t - \pi^m \Big(\frac{F''(\tw_t)}{F'(\tw_t)^2}\Big) b^m(t,\pi^m(\tw_t))\Big\|^2 \\
	& \leq 4 \bigg[ \bigg( \int_{-\infty}^{-L^m} + \int_{L^m}^{\infty} \bigg) \bigg( \frac{F''(\tw_t)(y)}{F'(\tw_t)^2(y)} \partial_t \tw_t(y) \bigg)^2 dy \\
	& \qquad + \Big\| \bigg( \frac{F''(\tw_t)}{F'(\tw_t)^2} - \pi^m \bigg(\frac{F''(\tw_t)}{F'(\tw_t)^2}\bigg)\bigg) \partial_t \tw_t \mathds{1}_{(-L^m, L^m)} \Big\|^2 \\
	& \qquad + \Big\| \pi^m \bigg(\frac{F''(\tw_t)}{F'(\tw_t)^2}\bigg) (\partial_t \tw_t - \pi^m(\partial_t \tw_t)) \mathds{1}_{(-L^m, L^m)}\Big\|^2\\
	& \qquad + \Big\| \pi^m \bigg(\frac{F''(\tw_t)}{F'(\tw_t)^2} \bigg) \bigg( \pi^m(w\ast F(\tw_t))- \pi^m(w\ast \pi^m(F(\tw_t))) \\
	& \qquad \qquad - \sum_k  \big( w^{m,+}_k  F(\tw_t(L^m)) + w^{m,-}_k F(\tw_t)(-L^m)\big) \mathds{1}_{I^m_k} \bigg) \Big\|^2 \bigg]\\
	& = 4 (S_{7,1} + S_{7,2} + S_{7,3} + S_{7,4}).
\end{align*}
We have
\begin{align*}
	S_{7,2}
	& \leq  \Big\| \frac{F^{(3)}(\twp) \twp_x}{(F'(\twp))^2} - \frac{2 (F''(\twp)^2\twp_x}{(F'(\twp))^3}\Big\|_{\infty}^2 \frac{1}{m^2} \| \partial_t \tw_t\|^2 
\end{align*}
and, as in (\ref{eq:convpim}),
\begin{align*}
S_{7,3} 
& \leq \Big\|\frac{F''(\tw_t)}{F'(\tw_t)^2}\Big\|_{\infty}^2 \|(\partial_t \tw_t - \pi^m(\partial_t \tw_t))\mathds{1}_{(-L^m, l^m)}\|^2 \\
& \leq \frac{1}{m^2} c^2 \Big\|\frac{F''(\tw_t)}{F'(\tw_t)^2}\Big\|_{\infty}^2 \|\twp_{xx}\|^2.
\end{align*}
The last summand satisfies, using (\ref{eq:riemannsum}) and (\ref{eq:convpim}),
\begin{align*}
	S_{7,4}
	& \leq 3 \Big\|\frac{F''(\tw_t)}{F'(\tw_t)^2}\Big\|_{\infty}^2  \\
	& \qquad \bigg[ \sum_k \frac{1}{m}\bigg( \int_{-L^m}^{L^m} w(\tfrac{k}{m}-y) ( F(\tw_t(y))-\pi^m(F(\tw_t))(y))^2 dy \bigg)^2 + S_3 + S_4\bigg] \\
	& \leq 3 \Big\|\frac{F''(\tw_t)}{F'(\tw_t)^2}\Big\|_{\infty}^2 (1+ \frac{1}{m}\|w_x\|_1) \|F'\|_{\infty}^2 \frac{1}{m^2} \|\twp_x\|^2 + S_3 + S_4).
\end{align*}
\subsection{The It\^o Correction}
\label{subsection:estimationitocorrection}

\begin{align*}
  & \|\sigma^m(t,v^m_t+ \pi^m(\tw_t)) \circ \Phi^m - \sigma(t,v_t+ \tw_t)\|_{L^0_2}^2\\
  & \leq 3 \Big( \|\left(\sigma^m(t,v^m_t+ \pi^m(\tw_t)) -  \sigma^m(t, (v_t + \tw_t) \mathds{1}_{(-L^m, L^m)})  \right) \circ \Phi^m\|_{L^0_2}^2 \\
	& \qquad {} + \|\left( \sigma^m(t, (v_t + \tw_t) \mathds{1}_{(-L^m, L^m)}) - \sigma(t, v_t + \tw_t) \right) \circ \Phi^m\|_{L^0_2}^2 \\
  & \qquad {} + \|\sigma(t, v_t + \tw_t)  \circ \Phi^m-\sigma(t, v_t+\tw_t)\|_{L^0_2}^2 \Big)\\
  & =: 3 (S_8 + S_9 + S_{10}).
\end{align*}
Let $(e_k)$ be an orthonormal basis of $L^2(\R)$.
Note that by Parseval's identity
\begin{align*}
	& \sum_k \left( \Phi^m(\sqrt{Q}e_k) \right)^2\\
	& = \sum_k \bigg( \sum_l 2m  \langle \sqrt{Q}e_k, \mathds{1}_{J^m_l} \rangle \mathds{1}_{I^m_l} \bigg)^2  = \sum_l \sum_k 4m^2 \langle \sqrt{Q}\mathds{1}_{J^m_l}, e_k \rangle^2\mathds{1}_{I^m_l} \\
	& = \sum_l 4 m^2  \|\sqrt{Q}\mathds{1}_{J^m_l}\|^2 \mathds{1}_{I^m_l} = \sum_l 4m^2 \int \bigg( \int_{J^m_l} q(x,y) dy \bigg)^2 dx  \ \mathds{1}_{I^m_l} \\
	& \leq \sum_l 2m \int \int_{J^m_l} q^2(x,y) dy dx  \ \mathds{1}_{I^m_l} \leq \sup_x \|q(x,\cdot)\|^2 \mathds{1}_{[-L^m, L^m)}.
\end{align*}
Thus, 
\begin{align*}
	S_8 
	& = \sum_k \int_{-\infty}^{\infty} \Big( \sigma^m(t,v^m_t+ \pi^m(\tw_t))\\
	& \qquad {} -  \sigma^m(t, (v_t + \tw_t) \mathds{1}_{(-L^m, L^m)})  \Big)^2(x) \left( \Phi^m(\sqrt{Q}e_k) \right)^2(x) dx\\
	& \leq \sup_x \|q(x,\cdot)\|^2 \ L_{\sigma}^2 \|v^m_t + \pi^m(\tw_t)-(v_t + \tw_t) \mathds{1}_{(-L^m, L^m)}\|^2\\
	& \stackrel{(\ref{eq:convpim})}{\leq} 2 \sup_x \|q(x,\cdot)\|^2 \ L_{\sigma}^2  \Big(\|v^m_t-v_t\|^2 + \frac{1}{m^2} \|\twp_x\|^2 \Big)
\end{align*}
and
\begin{align*}
S_9 & = \|\left( \sigma^m(t, (v_t + \tw_t) \mathds{1}_{(-L^m, L^m)}) - \sigma(t, v_t + \tw_t)  \right) \circ \Phi^m\|_{L^0_2}^2 \\
&  \leq \sup_x \|q(x,\cdot)\|^2 \|\sigma^m(t, (v_t + \tw_t) \mathds{1}_{(-L^m, L^m)}) - \sigma(t, v_t + \tw_t) \|^2.
\end{align*}
Using Parseval's identity again we get
\begin{align*}
	S_{10}
	& = \int_{-\infty}^{\infty} \sigma(t, v_t + \tw_t)^2(x) \sum_k  \bigg( \sum_l 2m \langle \sqrt{Q}e_k, \mathds{1}_{J^m_l} \rangle \mathds{1}_{I^m_l}(x) - \sqrt{Q}e_k(x) \bigg)^2 dx  \\
	& = \sum_l \int_{\frac{l}{m}}^{\frac{l+1}{m}} \sigma(t, v_t + \tw_t)^2(x) \sum_k  \bigg(2 m \int_{J^m_l} \int_{-\infty}^{\infty} (q(z,y)- q(x,y)) e_k(y) dy dz \bigg)^2 dx  \\
	& \qquad {} + \bigg( \int_{-\infty}^{-L^m} + \int_{L^m}^{\infty} \bigg) \sigma(t, v_t + \tw_t)^2(x) \sum_k (\sqrt{Q}e_k(x))^2 dx \\
	& \leq \sum_l \int_{\frac{l}{m}}^{\frac{l+1}{m}}\sigma(t, v_t + \tw_t)^2(x)  \bigg\|2m \int_{J^m_l} (q(z,\cdot)- q(x,\cdot))dz\bigg\|^2 dx  \\
	& \qquad {} +\bigg( \int_{-\infty}^{-L^m} + \int_{L^m}^{\infty} \bigg) \sigma(t, v_t + \tw_t)^2(x) \sum_k \left( \int_{-\infty}^{\infty} q(x,y) e_k(y) dy \right)^2 dx\\
	& \leq \|\sigma(t, v_t + \tw_t)\|^2 \sup_l \sup_{x \in I^m_l} \bigg\|2m \int_{J^m_l} q(z,\cdot) dz - q(x,\cdot)\bigg\|^2\\
		& \qquad {} + \sup_x \|q(x,\cdot)\|^2 \bigg( \int_{-\infty}^{-L^m} + \int_{L^m}^{\infty} \bigg) \sigma(t, v_t + \tw_t)^2(x) dx.
\end{align*}

\subsection{Application of Gronwall's Lemma}

We use $K, K_1, K_2, \tilde{K},$ etc. to denote suitable constants that may differ from step to step.
Summarizing the previous steps and using Young's inequality we arrive at
\begin{align*}
	& \frac{1}{2} d\|v^m_t-v_t\|^2 \\
	& \leq \bigg[- \|v^m_t-v_t\|^2 + \frac{1}{2} \|v^m_t-v_t\|^2 \\
	& \qquad {} + \frac{1}{2} \Big\|v^m_t - v_t + b^m(t, v^m_t+\pi^m(\tw_t)) - b(t, v_t+\tw_t)  - \pi^m(\partial_t\tw_t) + \partial_t \tw_t\Big\|_2^2  \\
	& \qquad {} + \frac{1}{4N}  \|v^m_t-v_t\|^2 + \frac{1}{4N} \Big\|\frac{F''(\tw_t)}{F'(\tw_t)^2} \partial_t \tw_t - \pi^m \Big(\frac{F''(\tw_t)}{F'(\tw_t)^2}\Big) b^m(t,\pi^m(\tw_t))\Big\|^2 \bigg] dt\\
	& \qquad {} + \frac{1}{2} \|\sigma^m(t,v^m_t+ \pi^m(\tw_t)) \circ \Phi^m - \sigma(t,v_t+ \tw_t)\|_{L^0_2}^2 dt \\ 
	& \qquad {} + \langle v^m_t-v_t, \left( \sigma^m(t, v^m_t+\pi^m(\tw_t)) \circ \Phi^m - \sigma(t, v_t+\tw_t)\right) d\mathcal{W}^Q_t \rangle \\
	& \leq K_1 \|v_t^m-v_t\|^2 dt + K_2 R(t,v_t,m) dt + dM_t,
\end{align*}
where
\begin{align*}
	R(t,v_t,m) &  = \frac{1}{m^2} \left( \|v_t\|^2 + \|\hat{u}_x\|^2 + \|\hat{u}_{xx}\|^2\right)  + \left( \int_{L^m}^{\infty} + \int_{-\infty}^{-L^m} \right) \Big( v_t^2(x) + (w\ast|v_t|)^2(x) \\
	& \qquad {} + \sigma(t,v_t+\tw_t)^2(x) + (\partial_x \tw_t(x))^2 \Big) dx \\
	& \qquad {} + \|\sigma(t, v_t+\tw_t)\|^2 \sup_k \sup_{x \in I^m_k} \|2m \sqrt{Q} \mathds{1}_{J^m_k} - q(x,\cdot)\|^2 \\
	& \qquad {}	+ \|\sigma^m(t, (v_t+\tw_t) \mathds{1}_{(-L^m, L^m)})-\sigma(t,v_t+\tw_t)\|^2
\end{align*}
and
\[M_t = \int_0^t \langle v^m_s-v_s, \left(\sigma^m(s, v^m_s+\pi^m(\tw_s)) \circ \Phi^m - \sigma(s, v_s+\tw_s)\right) d\mathcal{W}^Q_s \rangle \]
is a martingale
with quadratic variation process
\begin{equation}
\label{eq:quadraticvariation}
\begin{aligned}
 \left[M\right]_t &
 = \int_0^t \sum_k \langle v^m_s-v_s, \Big(\sigma^m(s, v^m_s+\pi^m(\tw_s)) \circ \Phi^m \\
& \qquad {} - \sigma(s, v_s+\tw_s)\Big) \circ \sqrt{Q} e_k \rangle^2 ds \\
& \leq \int_0^t \|v^m_s - v_s\|^2 \Big\|\Big( \sigma^m(s, v^m_s+\pi^m(\tw_s)) \circ \Phi^m \\
& \qquad {} - \sigma(s, v_s+\tw_s)\Big)\Big\|_{L_2^0}^2 ds.
\end{aligned}
\end{equation}
Applying It\^o's formula to the real-valued stochastic process $\|v^m_t-v_t\|^2$ we obtain for $p\geq 2$
\begin{align*}
	& d\|v^m_t-v_t\|^p \\*
	&=  \frac{p}{2}\|v^m_t-v_t\|^{p-2} d \|v^m_t-v_t\|^2 + \frac{p (p-2)}{8} \|v^m_t-v_t\|^{p-4} d\left[ \|v^m_t-v_t\|^2 \right]_t\\
	& \stackrel{(\ref{eq:quadraticvariation})}{\leq} K_1 p \|v^m_t-v_t\|^p dt + K_2 p \|v^m_t-v_t\|^{p-2} R(t,v_t,m) dt + p \|v^m_t-v_t\|^{p-2} dM_t \\
	& \qquad {} + \frac{p(p-2)}{2} \|v^m_t-v_t\|^{p-2}  \|\left( \sigma^m(t, v^m_t+\pi^m(\tw_t)) \circ \Phi^m - \sigma(t, v_t+\tw_t)\right) \|_{L_2^0}^2 dt.
\end{align*}
Estimating the last term as above and using Young's inequality we obtain
\begin{align*}
& d\|v^m_t-v_t\|^p   \\
&\leq \tilde{K_1} \|v^m_t-v_t\|^p dt  + \tilde{K_2} \|v^m_t-v_t\|^{p-2} R(t,v_t,m) dt  + p \|v^m_t-v_t\|^{p-2} dM_t\\
& \leq \left(\tilde{K_1}+\tilde{K_2} \tfrac{p-2}{p}\right) \|v^m_t-v_t\|^p dt  + \tilde{K_2}\tfrac{2}{p} R(t,v_t,m)^{\frac{p}{2}} dt+ p \|v^m_t-v_t\|^{p-2} dM_t.
\end{align*}
Integrating, maximizing over $t\leq T$, and taking expectations we get
\begin{equation}
\label{eq:estimatesup}
\begin{split}
	& E \left( \sup_{t\leq T} \|v^m_t-v_t\|^p\right) \\
  & \leq \|v^m_0-v_0\|^p + \left(\tilde{K_1}+\tilde{K_2} \frac{p-2}{p}\right) E \left( \sup_{t\leq T} \int_0^t \|v^m_s-v_s\|^p ds\right) \\
	& \qquad {} + \tilde{K_2}\frac{2}{p} E \left(\sup_{t\leq T} \int_0^t R(s,v_s,m)^{\frac{p}{2}} ds\right) + p E \left(\sup_{t\leq T} \int_0^t \|v^m_s-v_s\|^{p-2} dM_s\right)\\
	& \leq \|v^m_0-v_0\|^p + \left(\tilde{K_1}+\tilde{K_2} \frac{p-2}{p}\right) \int_0^T E \left( \sup_{s\leq t} \|v^m_s-v_s\|^p \right) dt \\
	& \qquad {} + \tilde{K_2}\frac{2}{p} E  \left( \int_0^T R(s,v_s,m)^{\frac{p}{2}} ds\right) + p E \left(\sup_{t\leq T} \int_0^t \|v^m_s-v_s\|^{p-2} dM_s\right).
\end{split}
\end{equation}
We estimate the last term using the Burkholder-Davis-Gundy inequality, (\ref{eq:quadraticvariation}), and Young's inequality:
\begin{align*}
& p E \left(\sup_{t\leq T} \int_0^t \|v^m_s-v_s\|^{p-2} dM_s \right) \\
& \leq K E \left( \int_0^T \|v_s^m-v_s\|_2^{2p-2} \|\left( \sigma^m(s, v^m_s+\pi^m(\tw_s)) \circ \Phi^m - \sigma(s, v_s+\tw_s)\right) \|_{L_2^0}^2 ds\right)^{\frac{1}{2}} \\
& \leq K E \bigg( \sup_{t\leq T} \|v_t^m-v_t\|^{p-1} \\
& \qquad \bigg( \int_0^T  \|\left( \sigma^m(s, v^m_s + \pi^m(\tw_s)) \circ \Phi^m - \sigma(s, v_s + \tw_s)\right) \|_{L_2^0}^2 ds \bigg)^{\frac{1}{2}} \bigg) \\
& \leq \frac{1}{2} E \left( \sup_{t\leq T} \|v_t^m-v_t\|^p \right)\\
& \qquad {} + \tilde{K}_3 E \left( \int_0^T  \|\left( \sigma^m(s, v^m_s +\pi^m(\tw_s)) \circ \Phi^m - \sigma(s, v_s+\tw_s)\right) \|_{L_2^0}^2 ds \right)^{\frac{p}{2}}.
\end{align*}
Bringing the first summand to the left-hand side of (\ref{eq:estimatesup}) this implies that
\begin{align*}
	& E \left( \sup_{t\leq T} \|v^m_t-v_t\|^p \right) \\
  & \leq 2 \|v^m_0-v_0\|^p + 2 \left(\tilde{K_1}+\tilde{K_2} \frac{p-2}{p}\right) \int_0^T \
	E \left( \sup_{s\leq t} \|v^m_s-v_s\|^p \right)dt \\
	& \qquad {} + \tilde{K_2}\frac{4}{p} E \left( \int_0^T R(s,v_s,m)^{\frac{p}{2}} ds \right)\\
	& \qquad {}+ 2 \tilde{K_3} E \left( \int_0^T  \|\left( \sigma^m(t, v^m_t+\pi^m(\tw_t)) \circ \Phi^m - \sigma(t, v_t+\tw_t)\right)\|_{L_2^0}^2 ds \right)^{\frac{p}{2}}.
\end{align*}
We estimate the last term as before and obtain
\begin{align*}
	& E \left( \int_0^T  \|\left(\sigma^m(s, v^m_s + \pi^m(\tw_s)) \circ \Phi^m - \sigma(s, v_s+\tw_s)\right) \|_{L_2^0}^2 ds \right)^{\frac{p}{2}} \\
	& \leq K E \left( \int_0^T \|v^m_s-v_s\|^2 + R(s, v_s,m) ds \right)^{\frac{p}{2}} \\
	& \leq \tilde{K}  E \left( \int_0^T \sup_{s\leq t} \|v^m_s-v_s\|^p dt + \int_0^T R(s, v_s,m)^{\frac{p}{2}} ds \right).
\end{align*}
Altogether we arrive at
\begin{align*}
	 E \left( \sup_{t\leq T} \|v^m_t-v_t\|^p\right) & \leq 2 \|v^m_0 -v_0\|_2^p + \hat{K_1} E \int_0^T R(t,v_t,m)^{\frac{p}{2}} dt \\*
	& \qquad {} + \hat{K_2} \int_0^T E \left( \sup_{s\leq t} \|v^m_s-v_s\|^p \right) dt.
\end{align*}
An application of Gronwall's Lemma yields
\begin{align*}
	E \left( \sup_{t\leq T} \|v^m_t-v_t\|^p\right) & \leq K \left( \|v^m_0-v_0\|^p  + E \left( \sup_{t\leq T} R(t,v_t,m)^{\frac{p}{2}} \right) \right).
\end{align*}
The sequence of continuous functions $f^m: [0,T] \rightarrow \R$, 
\begin{align*}
f^m(t) & = \bigg( \left( \int_{L^m}^{\infty} + \int_{-\infty}^{-L^m} \right) v_t^2(x) + (w\ast|v_t|)^2(x) + \sigma(t,v_t+\tw_t)^2(x) + (\partial_x \tw_t(x))^2  dx \bigg)^{\frac{p}{2}}
\end{align*}
is decreasing and converges pointwise to $0$ since all the integrands are in $L^2(\R)$. By Dini's Theorem the convergence is uniform. This together with the facts that\\ ${\|\sigma(t,v_t)\|_2^2 \leq K(1+\|v_t\|^2)}$ and ${E \left( \sup_{t\leq T}  \|v_t\|^2 \right) < \infty}$ by Proposition \ref{prop:existencesnfe}, assumptions (i) and (ii), and dominated convergence implies that
\begin{align*}
& E \left( \sup_{t\leq T} R(t,v_t,m)^{\frac{p}{2}} \right) \\
& \leq K \bigg(  \frac{1}{m^p} \left( E \left( \sup_{t\leq T}  \|v_t\|^p \right) + \|\hat{u}_x\|^p + \|\twp_{xx}\|^p\right) + E \left( \sup_{t\leq T} f^m(t) \right) \\
& \qquad {} + E \left( \sup_{t\leq T}  \|\sigma^m(t, (v_t+\tw_t)\mathds{1}_{(-L^m, L^m)}) \circ \Phi^m - \sigma(t, v_t+\tw_t)\|^p \right) \\
& \qquad {}  + E \left( \sup_{t\leq T}  \|\sigma(t,v_t+\tw_t)\|^p \right) \sup_k \sup_{x\in I^m_k} \|2m \sqrt{Q}\mathds{1}_{J^m_k} - q(x,\cdot)\|^p \bigg) \xrightarrow{m\rightarrow\infty}0
\end{align*}
and hence
\[	E \left( \sup_{t\leq T} \|v^m_t-v_t\|^p\right) \xrightarrow{m\rightarrow\infty}0.\]

\section*{Declarations}

\subsection*{Acknowledgements}
The work of E. Lang was supported by the DFG RTG 1845 and partially supported by the BMBF, FKZ01GQ1001B. The work of W. Stannat was supported by the BMBF, FKZ01GQ1001B.

\subsection*{Competing Interests} 
The authors declare that they have no competing interests.


\end{document}